\documentclass[graybox,footinfo]{svmult}

\smartqed
\usepackage{mathptmx}       
\usepackage{helvet}         
\usepackage{courier}        
\usepackage{type1cm}        
\usepackage{graphicx}       
\usepackage{epstopdf}

\usepackage{array,colortbl}
\usepackage{amsmath,amsfonts,amssymb,bm} 
\DeclareFontFamily{U}{mathx}{\hyphenchar\font45}
\DeclareFontShape{U}{mathx}{m}{n}{
      <5> <6> <7> <8> <9> <10>
      <10.95> <12> <14.4> <17.28> <20.74> <24.88>
      mathx10
      }{}
\DeclareSymbolFont{mathx}{U}{mathx}{m}{n}
\DeclareFontSubstitution{U}{mathx}{m}{n}
\DeclareMathAccent{\widecheck}      {0}{mathx}{"71}

\DeclareSymbolFont{bbold}{U}{bbold}{m}{n}
\DeclareSymbolFontAlphabet{\mathbbold}{bbold}

\usepackage{microtype} 

\usepackage[colorlinks=true,linkcolor=black,citecolor=black,urlcolor=black]{hyperref}
\usepackage{bookmark}
\usepackage{subfigure}
\makeatletter 
  \providecommand*{\toclevel@author}{999}
  \providecommand*{\toclevel@title}{0}
\makeatother

\begin{document}
\title*{Derivative-based Global Sensitivity Measures and Their Link with Sobol' Sensitivity Indices}
\author{Sergei Kucherenko \and Shugfang Song}

\institute{ Sergei Kucherenko \and Shugfang Song
\at Imperial College London, London, SW7 2AZ, UK
\email{s.kucherenko@imperial.ac.uk},
\email{shufangsong@nwpu.edu.cn}
}
\maketitle

\abstract{The variance-based method of Sobol' sensitivity indices is very popular
among practitioners due to its efficiency and easiness of interpretation.
However, for high-dimensional models the direct application of this method
can be very time-consuming and prohibitively expensive to use. One of the
alternative global sensitivity analysis methods known as the method of
derivative based global sensitivity measures (DGSM) has recently become
popular among practitioners. It has a link with the Morris screening method
and Sobol' sensitivity indices. DGSM are very easy to implement and evaluate
numerically. The computational time required for numerical evaluation of
DGSM is generally much lower than that for estimation of Sobol' sensitivity
indices. We present a survey of recent advances in DGSM and new results
concerning new lower and upper bounds on the values of Sobol' total
sensitivity indices $S_i^{tot} $. Using these bounds it is possible in most
cases to get a good practical estimation of the values of $S_i^{tot} $.
Several examples are used to illustrate an application of DGSM.}

\noindent \textbf{Keywords:} Global sensitivity analysis; Monte Carlo methods; Quasi Monte Carlo methods; Derivative based global measures; Morris method; Sobol’ sensitivity indices

\section{Introduction}\label{sec:1}

Global sensitivity analysis (GSA) is the study of how the uncertainty in the
model output is apportioned to the uncertainty in model inputs
\cite{Salt2008},\cite{Sobol2005}. GSA can provide
valuable information regarding the dependence of the model output to its
input parameters. The variance-based method of global sensitivity indices
developed by Sobol' \cite{Sobol1993} became very popular among
practitioners due to its efficiency and easiness of interpretation. There
are two types of Sobol' sensitivity indices: the main effect indices, which
estimate the individual contribution of each input parameter to the output
variance, and the total sensitivity indices, which measure the total
contribution of a single input factor or a group of inputs \cite{Homma1996}. The total
sensitivity indices are used to identify non-important variables which can
then be fixed at their nominal values to reduce model complexity
\cite{Salt2008}. For high-dimensional models the direct
application of variance-based GSA measures can be extremely time-consuming
and impractical.

A number of alternative SA techniques have been proposed. In this paper we
present derivative based global sensitivity measures (DGSM) and their link
with Sobol' sensitivity indices. DGSM are based on averaging local
derivatives using Monte Carlo or Quasi Monte Carlo sampling methods. These
measures were briefly introduced by Sobol' and Gershman in
\cite{Sobol1995}. Kucherenko \textit{et al }\cite{Kuch2009} introduced some other
derivative-based global sensitivity measures (DGSM) and coined the acronym
DGSM. They showed that the computational cost of numerical evaluation of
DGSM can be much lower than that for estimation of Sobol' sensitivity
indices which later was confirmed in other works
\cite{Kiparis2009}. DGSM can be seen as a generalization and
formalization of the Morris importance measure also known as elementary
effects \cite{Morris1991}. Sobol' and Kucherenko$^{
}$\cite{Sobol2009} $^{ }$proved theoretically that there is a
link between DGSM and the Sobol' total sensitivity index $S_i^{tot} $ for
the same input. They showed that DGSM can be used as an upper bound on total
sensitivity index $S_i^{tot} $. They also introduced modified DGSM which can
be used for both a single input and groups of inputs
\cite{Sobol2010}. Such measures can be applied for problems
with a high number of input variables to reduce the computational time.
Lamboni \textit{et al} \cite{Lamboni2013} extended results of Sobol' and
Kucherenko for models with input variables belonging to the class of
Boltzmann probability measures.

The numerical efficiency of the DGSM method can be improved by using the
automatic differentiation algorithm for calculation DGSM as was shown in
\cite{Kiparis2009}. However, the number of required function
evaluations still remains to be proportional to the number of inputs. This
dependence can be greatly reduced using an approach based on algorithmic
differentiation in the adjoint or reverse mode \cite{Griewank2008}.
It allows estimating all derivatives at a cost at most 4-6 times of that for
evaluating the original function \cite{Jansen2014}.

This paper is organised as follows: Section 2 presents Sobol' global
sensitivity indices. DGSM and lower and upper bounds on\textbf{
}total Sobol' sensitivity indices for uniformly distributed variables and
random variables are presented in Sections 3 and 4, respectively. In Section
5 we consider test cases which illustrate an application of DGSM and their
links with total Sobol' sensitivity indices. Finally, conclusions are
presented in Section 6.

\section{Sobol' global sensitivity indices}\label{sec:2}

The method of global sensitivity indices developed by Sobol' is based on
ANOVA decomposition \cite{Sobol1993}. Consider the square
integrable function $f({\rm {\bf x}})$ defined in the unit hypercube
$H^d=[0,1]^d$. The decomposition of $f({\rm {\bf x}})$

\begin{equation}
\label{eq1}
f({\bf{x}})=f_0+\sum\limits_{i=1}^d {f_i(x_i )}+\sum\limits_{i=1}^d {\sum\limits_{j > i}^d {f_{ij}(x_i ,x_j )} } + \cdots + f_{12 \cdots d}(x_1,\cdots,x_d ),
\end{equation}

\noindent
where $f_0  = \int_{H^d } {f(x)dx}$ , is called ANOVA if conditions

\begin{equation}
\label{eq1ort}
\int_{H^d } {f_{i_1 ...i_s } dx_{i_k } }  = 0
\end{equation}

\noindent
are satisfied for all different groups of indices $x_1,\cdots,x_s$ such that $1 \le i_1  < i_2  < ... < i_s  \le d $. These conditions guarantee that all
terms in (\ref{eq1}) are mutually orthogonal with respect to integration.

The variances of the terms in the ANOVA decomposition add up to the total variance:

\[
D = \int_{H^d } {f^2 ({\bf{x}})d} {\bf{x}} - f_0^2  = \sum\limits_{s = 1}^d {\sum\limits_{i_1  <  \cdot  \cdot  \cdot  < i_s }^d {D_{i_1 ...i_s }^{} } },
\]

\noindent
where $D_{i_1 ...i_s }^{}  =\int_{H^d } {f_{i_1 ...i_s }^2 (x_{i_1 } ,...,x_{i_s } )dx_{i_1 } ,...,x_{i_s } }$ are called
partial variances.

Total partial variances account for the total influence of the factor $x_i$:

\[
D_i^{tot}  = \sum\limits_{ < i > } {D_{i_1 ...i_s } },
\]

\noindent
where the sum $\sum\limits_{ < i > } {} $ is extended over all different groups of indices $x_1,\cdots,x_s$ satisfying condition $1 \le i_1  < i_2  < ... < i_s  \le n $, $1 \le s \le n$, where one of the indices is equal to $i$. The corresponding total sensitivity index is defined as

\[
S_i^{tot}  = {{D_i^{tot} } \mathord{\left/
 {\vphantom {{D_i^{tot} } D}} \right.
 \kern-\nulldelimiterspace} D}.
\]

Denote $u_i ({\bf{x}})$ the sum of
all terms in ANOVA decomposition (\ref{eq1}) that depend on $x_i$:

\[
u_i ({\bf{x}}) = f_i (x_i ) + \sum\limits_{j = 1,j \ne i}^d {f_{ij} (x_i ,x_j )}  +  \cdots  + f_{12 \cdots d} (x_1 , \cdots ,x_d ).
\]

\noindent
From the definition of ANOVA decomposition it follows that

\begin{equation}
\label{eq2}
\int_{H^d } {u_i ({\bf{x}})d{\bf{x}}}  = 0.
\end{equation}

The total partial variance $D_i^{tot}$ can be computed as
\[
D_i^{tot}  = \int_{H^d } {u_i^2 ({\bf{x}})d{\bf{x}}}  = \int_{H^d } {u_i^2 (x_i ,{\bf{z}})dx_i d{\bf{z}}}.
\]

Denote ${\bf{z}} = (x_1 ,...,x_{i - 1} ,x_{i + 1} ,...,x_d )$ the vector
of all variables but $x_i$, then ${\bf{x}} \equiv (x_i ,{\bf{z}})$ and
$f({\bf{x}}) \equiv f(x_i ,{\bf{z}})$. The ANOVA
decomposition of $f({\rm {\bf x}})$ in
(\ref{eq1}) can be presented in the following form

\[
f({\bf{x}}) = u_i (x_i ,{\bf{z}}) + v({\bf{z}}),
\]

\noindent
where $v({\bf{z}})$ is the sum of terms independent of $x_i$. Because of
(\ref{eq1ort}) and  (\ref{eq2}) it is easy to show that
$v({\bf{z}}) = \int_{H^d } {f({\bf{x}})dx_i }$. Hence

\begin{equation}
\label{eq3}
u_i (x_i ,{\bf{z}}) = f({\bf{x}}) - \int_{H^d } {f({\bf{x}})dx_i }.
\end{equation}

Then the total sensitivity index $S_i^{tot}$ is equal to

\begin{equation}
\label{eq4}
S_i^{tot}  = \frac{{\int_{H^d } {u_i^2 ({\bf{x}})d{\bf{x}}} }}{D}.
\end{equation}

We note that in the case of independent random variables all definitions
of the ANOVA decomposition remain to be correct but all derivations
should be considered in probabilistic sense as shown in \cite{Sobol2005}
and presented in Section 4.

\section{DGSM for uniformly distributed variables}\label{sec:3}

Consider continuously differentiable function $f({\rm {\bf x}})$ defined in the unit hypercube $H^d=[0,1]^d$ such that ${{\partial f} \mathord{\left/
 {\vphantom {{\partial f} {\partial x_i }}} \right.
 \kern-\nulldelimiterspace} {\partial x_i }} \in L_2 $.

\begin{theorem}
Assume that $c \le \left| {\frac{{\partial f}}{{\partial x_i }}} \right| \le C$. Then

\begin{equation}
\label{eq5}
\frac{{c^2 }}{{12D}} \le S_i^{tot}  \le \frac{{C^2 }}{{12D}}.
\end{equation}

\end{theorem}

\noindent
The proof is presented in \cite{Sobol2009}.

The Morris importance measure also known as elementary
effects originally defined as finite differences averaged over
a finite set of random points \cite{Morris1991} was generalized in  \cite{Kuch2009}:

\begin{equation}
\label{eq6'}
\mu_i  = \int_{H^d }  \left| {\frac{{\partial f({\bf{x}})}}{{\partial x_i }}} \right| d{\bf{x}}.
\end{equation}

\noindent
Kucherenko \textit{et al } \cite{Kuch2009} also introduced a new DGSM
measure:

\begin{equation}
\label{eq6}
\nu _i  = \int_{H^d } {\left( {\frac{{\partial f({\bf{x}})}}{{\partial x_i }}} \right)} ^2 d{\bf{x}}.
\end{equation}

\noindent
In this paper we define two new DGSM measures:

\begin{equation}
\label{eq7}
w_i^{(m)}  = \int_{H^d } {x_i^m \frac{{\partial f({\bf{x}})}}{{\partial x_i }}d} {\bf{x}},
\end{equation}

\noindent
where $m$ is a constant, $m > 0$,

\begin{equation}
\label{eq8}
\varsigma _i  = \frac{1}{2}\int_{H^d } {x_i (1 - x_i )\left( {\frac{{\partial f({\bf{x}})}}{{\partial x_i }}} \right)^2 d} {\bf{x}}.
\end{equation}

We note that $\nu _i $ is in
fact the mean value of $\left( {{{\partial f} \mathord{\left/
 {\vphantom {{\partial f} {\partial x_i }}} \right.
 \kern-\nulldelimiterspace} {\partial x_i }}} \right)^2 $. We also note that

\begin{equation}
\label{eq9}
\frac{{\partial f}}{{\partial x_i }} = \frac{{\partial u_i}}{{\partial x_i }}.
\end{equation}

\subsection{{Lower bounds on $S_i^{tot}$}}

\begin{theorem}

There exists the following lower bound between DGSM (\ref{eq6})
and the Sobol' total sensitivity index

\begin{equation}
\label{eq10}
\frac{{\left( {\int_{H^d } {\left[ {f\left( {1,{\bf{z}}} \right) - f\left( {0,{\bf{z}}} \right)} \right]\left[ {f\left( {1,{\bf{z}}} \right) + f\left( {0,{\bf{z}}} \right) - 2f\left( {\bf{x}} \right)} \right]d} {\bf{x}}} \right)^2 }}{{4\nu _i D}} < S_i^{tot}.
\end{equation}

\end{theorem}

\begin{proof}
Consider an integral

\begin{equation}
\label{eq11}
\int_{H^d } {u_i ({\bf{x}})\frac{{\partial u_i ({\bf{x}})}}{{\partial x_i }}} d{\bf{x}}.
\end{equation}

Applying the Cauchy--Schwarz inequality we obtain the following result:

\begin{equation}
\label{eq12}
\left( {\int_{H^d } {u_i ({\bf{x}})\frac{{\partial u_i ({\bf{x}})}}{{\partial x_i }}} d{\bf{x}}} \right)^2  \le \int_{H^d } {u_i^2 ({\bf{x}})} d{\bf{x}} \cdot \int_{H^d } {\left( {\frac{{\partial u_i ({\bf{x}})}}{{\partial x_i }}} \right)} ^2 d{\bf{x}}.
\end{equation}

It is easy to prove that the left and right parts of this inequality cannot
be equal. Indeed, for them to be equal functions $u_i ({\bf{x}})$ and $\frac{{\partial u_i ({\bf{x}})}}{{\partial x_i }}$
 should be linearly dependent. For simplicity
consider a one-dimensional case: $x \in [0,1]$. Let's assume

\[
\frac{{\partial u(x)}}{{\partial x}} = Au(x),
\]

\noindent
where $A$ is a constant. The general solution to this equation $u(x) = B\exp (Ax)$
, where $B$ is a constant. It is easy to see that this
solution is not consistent with condition (\ref{eq2}) which should be imposed on
function $u(x)$.

Integral $\int_{H^d } {u_i ({\bf{x}})\frac{{\partial u_i ({\bf{x}})}}{{\partial x_i }}} d{\bf{x}}$ can be transformed as

\begin{equation}
\label{eq13}
\begin{array}{l}
 \int_{H^d } {u_i ({\bf{x}})\frac{{\partial u_i ({\bf{x}})}}{{\partial x_i }}} d{\bf{x}} = \frac{1}{2}\int_{H^d } {\frac{{\partial u_i^2 ({\bf{x}})}}{{\partial x_i }}} d{\bf{x}} \\
  = \frac{1}{2}\int_{H^{d - 1} } {\left( {u_i^2 (1,{\bf{z}}) - u_i^2 (0,{\bf{z}})} \right)d{\bf{z}}}  \\
  = \frac{1}{2}\int_{H^{d - 1} } {\left( {u_i (1,{\bf{z}}) - u_i (0,{\bf{z}})} \right)\left( {u_i (1,{\bf{z}}) + u_i (0,{\bf{z}})} \right)d{\bf{z}}}  \\
  = \frac{1}{2}\int_{H^d } {\left( {f(1,{\bf{z}}) - f(0,{\bf{z}})} \right)\left( {f(1,{\bf{z}}) + f(0,{\bf{z}}) - 2v({\bf{z}})} \right)d{\bf{z}}.}  \\
 \end{array}
\end{equation}

All terms in the last integrand are independent of $x_i$, hence we can replace integration with
respect to $d{\bf{z}}$ to integration with respect to $d{\bf{x}}$ and substitute $v({\bf{z}})$ for $f( {\bf{x}} )$ in
the integrand due to condition (\ref{eq2}). Then (\ref{eq13}) can be presented as

\begin{equation}
\label{eq14}
\int_{H^d } {u_i ({\bf{x}})\frac{{\partial u_i ({\bf{x}})}}{{\partial x_i }}} d{\bf{x}} = \frac{1}{2}\int_{H^d } {\left[ {f\left( {1,{\bf{z}}} \right) - f\left( {0,{\bf{z}}} \right)} \right]\left[ {f\left( {1,{\bf{z}}} \right) + f\left( {0,{\bf{z}}} \right) - 2f\left( {\bf{x}} \right)} \right]d} {\bf{x}}
\end{equation}

From (\ref{eq9}) $\frac{{\partial u_i ({\bf{x}})}}{{\partial x_i }} = \frac{{\partial f({\bf{x}})}}{{\partial x_i }}$, hence
the right hand side of (\ref{eq12}) can be written as $\nu _i D_i^{tot} $.
Finally dividing (\ref{eq12}) by $\nu _i D$ and using (\ref{eq14}), we obtain the lower bound (\ref{eq10}).
  \qed
\end{proof}

We call
\[
\frac{{\left( {\int_{H^d } {\left[ {f\left( {1,{\bf{z}}} \right) - f\left( {0,{\bf{z}}} \right)} \right]\left[ {f\left( {1,{\bf{z}}} \right) + f\left( {0,{\bf{z}}} \right) - 2f\left( {\bf{x}} \right)} \right]d} {\bf{x}}} \right)^2 }}{{4\nu _i D}}
\]
\noindent
the lower bound number one (LB1).

\begin{theorem}

There exists the following lower bound between DGSM (\ref{eq7})
and the Sobol' total sensitivity index

\begin{equation}
\label{eq15}
\frac{{(2m + 1)\left[ {\int_{H^d } {\left( {f(1,{\bf{z}}) - f({\bf{x}})} \right)d{\bf{x}}}  - w_i^{(m + 1)} } \right]^2 }}{{(m + 1)^2 D}} < S_i^{tot}
\end{equation}

\end{theorem}


\begin{proof}
 Consider an integral

\begin{equation}
\label{eq16}
\int_{H^d } {x_i^m u_i ({\bf{x}})d} {\bf{x}}.
\end{equation}

Applying the Cauchy--Schwarz inequality we obtain the following result:

\begin{equation}
\label{eq17}
\left( {\int_{H^d } {x_i^m u_i ({\bf{x}})d} {\bf{x}}} \right)^2  \le \int_{H^d } {x_i^{2m} d} {\bf{x}} \cdot \int_{H^d } {u_i^2 ({\bf{x}})d} {\bf{x}}.
\end{equation}

It is easy to see that equality in (\ref{eq17}) cannot be attained. For this to
happen functions $u_i ({\bf{x}})$ and $x_i^m$ should be
linearly dependent. For simplicity consider a one-dimensional case:
$x \in [0,1]$. Let's assume

\[
u(x) = Ax^m ,
\]

\noindent
where $A\neq0$ is a constant. This solution does not satisfy condition (\ref{eq2}) which should be
imposed on function $u(x)$.

Further we use the following transformation:

\[
\int_{H^d } {\frac{{\partial \left( {x_i^{m + 1} u_i ({\bf{x}})} \right)}}{{\partial x_i }}d} {\bf{x}} = (m + 1)\int_{H^d } {x_i^m u_i ({\bf{x}})d} {\bf{x}} + \int_{H^d } {x_i^{m + 1} \frac{{\partial u_i ({\bf{x}})}}{{\partial x_i }}d} {\bf{x}}
\]

\noindent
to present integral (\ref{eq16}) in a form:

\begin{equation}
\label{eq18}
\begin{array}{l}
 \int_{H^d } {x_i^m u_i ({\bf{x}})d} {\bf{x}} = \frac{1}{{m + 1}}\left[ {\int_{H^d } {\frac{{\partial \left( {x_i^{m + 1} u_i ({\bf{x}})} \right)}}{{\partial x_i }}d} {\bf{x}} - \int_{H^d } {x_i^{m + 1} \frac{{\partial u_i ({\bf{x}})}}{{\partial x_i }}d} {\bf{x}}} \right] \\
  = \frac{1}{{m + 1}}\left[ {\int_{H^{d - 1} } {u_i (1,{\bf{z}})d} {\bf{z}} - \int_{H^d } {x_i^{m + 1} \frac{{\partial u_i ({\bf{x}})}}{{\partial x_i }}d} {\bf{x}}} \right] \\
  = \frac{1}{{m + 1}}\left[ {\int_{H^d } {\left( {f(1,{\bf{z}}) - f({\bf{x}})} \right)d{\bf{x}}}  - \int_{H^d } {x_i^{m + 1} \frac{{\partial u_i ({\bf{x}})}}{{\partial x_i }}d} {\bf{x}}} \right]. \\
 \end{array}
\end{equation}

\noindent
We notice that

\begin{equation}
\label{eq19}
\int_{H^d } {x_i^{2m} d} {\bf{x}} = \frac{1}{{(2m + 1)}}.
\end{equation}

Using (\ref{eq18}) and (\ref{eq19}) and dividing (\ref{eq17}) by $D$ we obtain (\ref{eq15}).
\qed

\end{proof}

This second lower bound on $S_i^{tot}$ we denote $\gamma (m)$:

\begin{equation}
\label{eq20}
\gamma (m) = \frac{{(2m + 1)\left[ {\int_{H^d } {\left( {f(1,{\bf{z}}) - f({\bf{x}})} \right)d{\bf{x}}}  - w_i^{(m + 1)} } \right]^2 }}{{(m + 1)^2 D}} < S_i^{tot}.
\end{equation}

In fact, this is a set of lower bounds depending on parameter $m$. We are
interested in the value of $m$ at which $\gamma (m)$ attains its maximum. Further we use star to denote
such a value $m$: $m^*  = \arg _{} \max (\gamma (m))$ and call

\begin{equation}
\label{eq21}
\gamma ^* (m^* ) = \frac{{(2m^*  + 1)\left[ {\int_{H^d } {\left( {f(1,{\bf{z}}) - f({\bf{x}})} \right)d{\bf{x}}}  - w_i^{(m^*  + 1)} } \right]^2 }}{{(m^*  + 1)^2 D}}
\end{equation}

\noindent
the lower bound number two (LB2).

We define the maximum lower bound $LB^* $ as

\begin{equation}
\label{eq22}
LB^*=max(LB1, LB2).
\end{equation}

We note that both lower and upper bounds can be estimated by a set of
derivative based measures:

\begin{equation}
\label{eq23}
\Upsilon _i  = \{ \nu _i ,w_i^{(m)} \} ,{\rm{  }}m > 0.
\end{equation}

\subsection{{Upper bounds on $S_i^{tot}$}}

\begin{theorem}

\begin{equation}
\label{eq24}
S_i^{tot}  \le \frac{{\nu _i }}{{\pi ^2 D}}.
\end{equation}

\end{theorem}

\noindent The proof of this Theorem in given in \cite{Sobol2009}.

Consider the set of values $\nu _1 ,...,\nu _n $, $1 \le i \le n$. One can
expect that smaller $\nu _i$ correspond to less influential variables $x_i$.

We further call (\ref{eq24}) the upper bound number one (UB1).

\begin{theorem}

\begin{equation}
\label{eq25}
S_i^{tot}  \le \frac{{\varsigma _i }}{D},
\end{equation}
\end{theorem}

\noindent
where $\varsigma _i$ is given by (\ref{eq8}).

\begin{proof}

We use the following inequality \cite{Hardy1973}:

\begin{equation}
\label{eq26}
0 \le \int_0^1 {u^2 dx}  - \left( {\int_0^1 {udx} } \right)^2  \le \frac{1}{2}\int_0^1 {x(1 - x)u'^2 dx}.
\end{equation}

\noindent
The inequality is reduced to an equality only if $u$ is constant. Assume that
$u$ is given by (\ref{eq2}), then $\int_0^1 {udx}  = 0$, and from (\ref{eq26})
we obtain (\ref{eq25}).
\qed
\end{proof}

Further we call $\frac{{\varsigma _i }}{D}$ the
upper bound number two (UB2). We note that $\frac{1}{2}x_i (1 - x_i )$ for
$0 \le x_i  \le 1$ is bounded: $0 \le \frac{1}{2}x_i (1 - x_i ) \le \frac{1}{8}$. Therefore, $0 \le \varsigma _i  \le \frac{1}{8}\nu _i$.

\subsection{{Computational costs}}

All DGSM can be computed using the same set of partial derivatives
$\frac{{\partial f(x)}}{{\partial x_i }},{\rm{ }}i = 1,...,d$. Evaluation of
$\frac{{\partial f(x)}}{{\partial x_i }}$ can be done
analytically for explicitly given easily-differentiable functions or
numerically.

In the case of straightforward numerical estimations of all partial
derivatives and computation of integrals using MC or QMC methods, the
number of required function evaluations for a set of all input variables is
equal to $N(d+1)$, where $N$ is
a number of sampled points. Computing LB1 also requires values of
$f\left( {0,z} \right),f\left( {1,z} \right)$, while computing
LB2 requires only values of $f\left( {1,z} \right)$. In total, numerical computation of $LB^*$ for all input variables would
require $N_F^{LB^*}  = N(d + 1) + 2Nd = N(3d + 1)$ function
evaluations. Computation of all upper bounds require $N_F^{UB}  = N(d + 1)$
 function evaluations. We recall that
the number of function evaluations required for computation of
$S_i^{tot}$ is
$N_F^{S}  = N(d + 1)$
\cite{Salt2010}. The number of sampled points $N$ needed to
achieve numerical convergence can be different for DGSM and
$S_i^{tot}$. It is generally
lower for the case of DGSM. The numerical efficiency of the DGSM method can
be significantly increased by using algorithmic differentiation in the
adjoint (reverse) mode \cite{Griewank2008}. This approach allows
estimating all derivatives at a cost at most 6 times of that for
evaluating the original function $f(x)$ \cite{Jansen2014}.
However, as mentioned above lower bounds also require computation of
$f\left( {0,z} \right),f\left( {1,z} \right)$ so  $N_F^{LB^*}$ would only be reduced to
 $N_F^{LB^*}  = 6N + 2Nd = N(2d + 6)$, while $N_F^{UB}$ would  be equal to $6N$.

\[
\]
\section{DGSM for random variables}\label{sec:4}

Consider a function $f\left( {x_1 ,...,x_d } \right)$, where $x_1 ,...,x_d$ are independent random variables with distribution functions
$F_1 \left( {x_1 } \right),...,F_d \left( {x_d } \right)$. Thus the point
${\bf{x}} = (x_1 ,...,x_d )$ is defined in the Euclidean space $R^d $ and
its measure is $dF_1 \left( {x_1 } \right) \cdot  \cdot  \cdot dF_d \left( {x_d } \right)$.

The following DGSM was introduced in \cite{Sobol2009}:

\begin{equation}
\label{eq27}
\nu _i  = \int_{R^d } {\left( {\frac{{\partial f({\bf{x}})}}{{\partial x_i }}} \right)^2 d} F({\bf{x}}).
\end{equation}

We introduce a new measure

\begin{equation}
\label{eq28}
w_i  = \int_{R^d } {\frac{{\partial f({\bf{x}})}}{{\partial x_i }}d} F({\bf{x}}).
\end{equation}

\subsection{{The lower bounds on $S_i^{tot}$ for normal variables}}

Assume that $x_i$ is normally distributed with the finite variance $\sigma _i^2 $
 and the mean value $\mu _i $.

\begin{theorem}

\begin{equation}
\label{eq29}
\frac{{\sigma _i^2 w_i^2 }}{D} \le S_i^{tot} .
\end{equation}

\end{theorem}

\begin{proof}

Consider $\int_{R^d } {x_i u_i ({\bf{x}})d} F({\bf{x}})$. Applying the Cauchy--Schwarz inequality we obtain

\begin{equation}
\label{eq30}
\left( {\int_{R^d } {x_i u_i ({\bf{x}})d} F({\bf{x}})} \right)^2  \le \int_{R^d } {x_i^2 d} F({\bf{x}}) \cdot \int_{R^d } {u_i^2 ({\bf{x}})d} F({\bf{x}}).
\end{equation}

\noindent
Equality in (\ref{eq30}) can be attained if functions $u_i ({\bf{x}})$ and $x_i$ are linearly dependent. For simplicity consider a
one-dimensional case. Let's assume
\[
u(x) = A(x - \mu ),
\]

\noindent
where $A\neq0$ is a constant. This solution satisfies condition (\ref{eq2}) for normally distributed
variable $x$ with the mean value $\mu$: $\int_{R^d } {u(x)dF(x)}  = 0$.

For normally distributed variables the following equality is true \cite{Hardy1973}:

\begin{equation}
\label{eq31}
\left( {\int_{R^d } {x_i u_i ({\bf{x}})d} F({\bf{x}})} \right)^2  = \int_{R^d } {x_i^2 d} F({\bf{x}}) \cdot \int_{R^d } {\frac{{\partial u_i^{} ({\bf{x}})}}{{\partial x_i }}d} F({\bf{x}}).
\end{equation}

By definition $\int_{R^d } {x_i^2 d} F({\bf{x}}) = \sigma _i^2$.
Using (\ref{eq30}) and (\ref{eq31}) and dividing the resulting inequality by
$D$ we obtain the lower bound (\ref{eq29}).
\qed
\end{proof}

\subsection{{The upper bounds on $S_i^{tot}$ for normal variables}}

The following Theorem 7 is a generalization of Theorem 1.

\begin{theorem}
Assume that $c \le \left| {\frac{{\partial f}}{{\partial x_i }}} \right| \le C$, then

\begin{equation}
\label{eq32}
\frac{{\sigma _i^2 c^2 }}{D} \le S_i^{tot}  \le \frac{{\sigma _i^2 C^2 }}{D}.
\end{equation}

\noindent
The constant factor $\sigma _i^2$ cannot be improved.
\end{theorem}

\begin{theorem}

\begin{equation}
\label{eq33}
S_i^{tot}  \le \frac{{\sigma _i^2 }}{D}\nu _i.
\end{equation}

\noindent
The constant factor $\sigma _i^2$ cannot be reduced.
\end{theorem}

 Proofs are presented in \cite{Sobol2009}.

\[
\]
\section{Test cases}\label{sec:5}

In this section we present the results of analytical and numerical
estimation of $S_i$, $S_i^{tot}$, LB1, LB2 and UB1,
UB2. The analytical values for DGSM and $S_i^{tot}$ were calculated and compared with numerical results.
For text case 2 we present convergence plots in the form of root mean square
error (RMSE) versus the number of sampled points $N$. To reduce the scatter in
the error estimation the values of RMSE were averaged over $K$ = 25 independent
runs:

\[
\varepsilon _i  = \left( {\frac{1}{K}\sum\limits_{k = 1}^K {\left( {\frac{{I_{i,k}^*  - I_0 }}{{I_0^{} }}} \right)^2 } } \right)^{\frac{1}{2}}.
\]

Here $I_i^*$ is numerically computed values of $S_i^{tot}$, LB1, LB2 or UB1, UB2, $I_0$ is the corresponding analytical value of
$S_i^{tot}$, LB1, LB2 or UB1,
UB2. The RMSE can be approximated by a trend line $cN^{ - \alpha }$. Values of ($- \alpha)$ are given in brackets on the plots.
QMC integration based on Sobol' sequences was used in all numerical tests.

\bigskip
\textbf{Example 1.} Consider a linear with respect to $x_i$ function:

\[
f({\bf{x}}) = a(z)x_i  + b(z).
\]

For this function $S_i=S_i^{tot}$, $D_i^{tot}  = \frac{1}{{12}}\int_{H^{d - 1} } {a^2 (z)dz} $,
$\nu _i  = \int_{H^{d - 1} } {a^2 (z)dz} $,
\medskip
$LB1 = \frac{{\left( {\int_{H^d } {\left( {a^2 (z) - 2a^2 (z)x_i } \right)dzdx_i } } \right)^2 }}{{4D\int_{H^{d - 1} } {a^2 (z)dz} }} = 0$ and $\gamma (m) = \frac{{(2m + 1)m^2 \left( {\int_{H^{d - 1} } {a(z)dz} } \right)^2 }}{{4(m + 2)^2 (m + 1)^2 D}}$.
\smallskip
A maximum value of $\gamma (m) $ is attained at $m^*
$=3.745, when $\gamma ^* (m^* ) = \frac{{0.0401}}{D}\left( {\int_{} {a(z)dz} } \right)^2 $.
\smallskip
The lower and upper bounds are $LB* \approx {\rm{0}}{\rm{.48}}S_i^{tot} $. $UB1 \approx 1.{\rm{22}}S_i^{tot} $. $UB2 = \frac{1}{{12D}}\int_0^1 {a(z)^2 dz}  = S_i^{tot}$. For this test function UB2 $<$ UB1.

\bigskip
\textbf{Example 2.} Consider the so-called g-function which is often used in
GSA for illustration purposes:

\[
f({\bf{x}}) = \prod\limits_{i = 1}^d {g_i },
\]

\noindent
where $g_i  = \frac{{|4x_i  - 2| + a_i }}{{1 + a_i }}$,
$a_i (i = 1,...,d)$ are constants. It
is easy to see that for this function $f_i (x_i ) = (g_i  - 1)$,
 $u_i ({\bf{x}}) = (g_i  - 1)\prod\limits_{j = 1,j \ne i}^d {g_j } $
 and as a result LB1=0. The total variance is $D =  - 1 + \prod\limits_{j = 1}^d {\left( {1 + \frac{{1/3}}{{(1 + a_j )^2 }}} \right)} $. The analytical values of
$S_i$, $S_i^{tot}$ and LB2 are given in Table 1.

\begin{table}[!hbp]
\caption{ The analytical expressions for $S_i$, $S_i^{tot}$ and LB2 for g-function}
\label{tab1}
\begin{tabular}
{|p{70pt}|p{120pt}|p{125pt}|}
\hline
$S_i$&
$S_i^{tot}$&
$\gamma (m)$ \\
\hline
$\displaystyle\frac{{1/3}}{{(1 + a_i )^2 D}}$&
$\displaystyle\frac{{\frac{{1/3}}{{(1 + a_i )^2 }}\prod\limits_{j = 1,j \ne i}^d {\left( {1 + \frac{{1/3}}{{(1 + a_j )^2 }}} \right)} }}{D} $&
$\displaystyle\frac{{(2m + 1)\left[ {1 - \frac{{4\left( {1 - (1/2)^{m + 1} } \right)}}{{m + 2}}} \right]^2 }}{{(1 + a_i )^2 (m + 1)^2 D}} $ \\
\hline
\end{tabular}
\end{table}

By solving equation $\frac{{d\gamma (m)}}{{dm}} = 0$
, we find that $m^*$=9.64, $\gamma (m^* ) = \frac{{0.0772}}{{(1 + a_i )^2 D}}$.
\smallskip
It is interesting to note that $m^*$ does not depend on $a_i$, $ i = 1,2,...,d$
and $d$. In the extreme cases:
\smallskip
 if $a_i  \to \infty $
 for all $i$, $\frac{{\gamma (m^* )}}{{S_i^{tot} }} \to 0.257$, $\frac{{S_i }}{{S_i^{tot} }} \to 1$,
 while if $a_i  \to 0 $ for all $i$, $\frac{{\gamma (m^* )}}{{S_i^{tot} }} \to \frac{{0.257}}{{(4/3)^{d - 1} }}$,
$\frac{{S_i }}{{S_i^{tot} }} \to \frac{1}{{(4/3)^{d - 1} }}$. The analytical
expression for $S_i^{tot}$, UB1
and UB2 are given in Table 2.

\begin{table}[!hbp]
\caption{The analytical expressions for $S_i^{tot}$ UB1 and UB2 for g-function}
\label{tab2}
\begin{tabular}
{|p{105pt}|p{105pt}|p{105pt}|}
\hline
$S_i^{tot}$&
$UB1$&
$UB2$ \\
\hline
$\displaystyle\frac{{\frac{{1/3}}{{(1 + a_i )^2 }}\prod\limits_{j = 1,j \ne i}^d {\left( {1 + \frac{{1/3}}{{(1 + a_j )^2 }}} \right)} }}{D}$&
$\displaystyle\frac{{16\prod\limits_{j = 1,j \ne i}^d {\left( {1 + \frac{{1/3}}{{(1 + a_j )^2 }}} \right)} }}{{(1 + a_i )^2 \pi ^2 D}}$&
$\displaystyle\frac{{4\prod\limits_{j = 1,j \ne i}^d {\left( {1 + \frac{{1/3}}{{(1 + a_j )^2 }}} \right)} }}{{3(1 + a_i )^2 D}}$\\
\hline
\end{tabular}
\end{table}

For this test function $\frac{{S_i^{tot} }}{{{\rm{UB1}}}} = \frac{{\pi ^2 }}{{48}}
$, $\frac{{S_i^{tot} }}{{{\rm{UB2}}}} = \frac{1}{4}$, hence
$\frac{{{\rm{UB2}}}}{{{\rm{UB1}}}} = \frac{{\pi ^2 }}{{12}} < 1$. Values of $S_i$, $S_i^{tot}$, UB and LB2 for the case of \textbf{\textit{a}}=[0,1,4.5,9,99,99,99,99], $d$=8 are given in Table 3
and shown in Figure 1. We can conclude that for this test function the knowledge of LB2 
and UB1, UB2 allows to rank correctly all the variables in the order of
their importance.

\begin{table}[!hbp]
\caption{Values of LB*, $S_i$, $S_i^{tot}$ , UB1 and
UB1. Example 2, \textbf{\textit{a}}=[0,1,4.5,9,99,99,99,99], $d$=8.}
\label{tab3}
\begin{tabular}
{|p{35pt}|p{50pt}|p{50pt}|p{60pt}|p{60pt}|p{60pt}|}
\hline
$ $&
$x_1$&
$x_2$&
$x_3$&
$x_4$&
$x_5...x_8$ \\
\hline
$LB^*$&
$0.166$&
$0.0416$&
$0.00549$&
$0.00166$&
$0.000017$\\
\hline
$S_i$&
$0.716$&
$0.179$&
$0.0237$&
$0.00720$&
$0.0000716$\\
\hline
$S_i^{tot}$&
$0.788$&
$0.242$&
$0.0343$&
$0.0105$&
$0.000105$\\
\hline
$UB1$&
$3.828$&
$1.178$&
$0.167$&
$0.0509$&
$0.000501$\\
\hline
$UB2$&
$3.149$&
$0.969$&
$0.137$&
$0.0418$&
$0.00042$\\
\hline
\end{tabular}
\end{table}


\begin{figure}[!]
\centering
\includegraphics[scale=0.7]{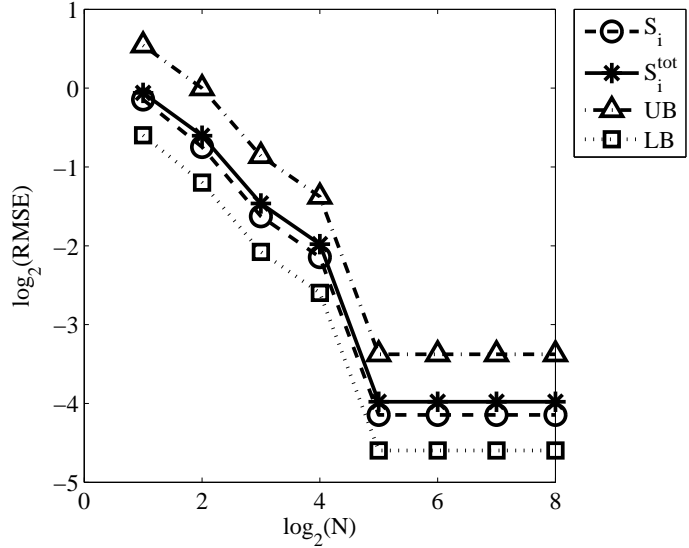}
\label{fig_sim_1}
\caption{ Values of $S_i$,$S_i^{tot}$, LB2 and UB1 for
all input variables. Example 2,
\textbf{\textit{a}}=[0,1,4.5,9,99,99,99,99], $d$=8.}
\end{figure}


Fig. 2 presents RMSE of numerical estimations of  $S_i^{tot}$, UB1 and LB2. For an individual input LB2 has the highest convergence rate, following by $S_i^{tot}$, and UB1 in terms of the number of sampled points. However, we recall that computation of all indices requires
$N_F^{LB*}  = N(3d + 1)$ function
evaluations for LB, while for $S_i^{tot}$ this number is $N_F^S  = N(d + 1)$ and for UB it is also $N_F^{UB}  = N(d + 1)$.

\begin{figure}[!]
\centering
{
\subfigure[] {\includegraphics[width=3.5cm,clip]{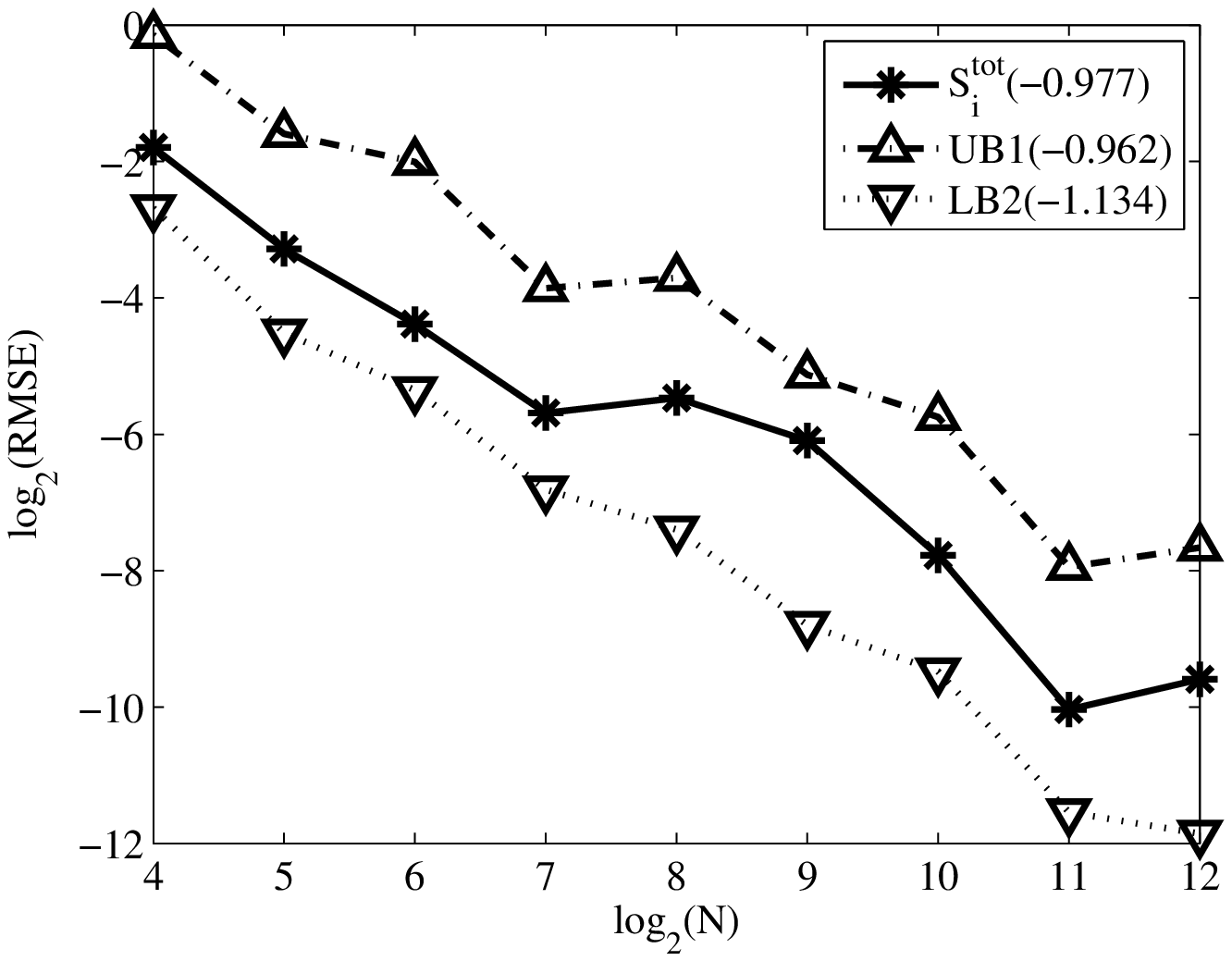}  }
\label{Fig.sub.1}
\subfigure [] {\includegraphics[width=3.5cm,clip]{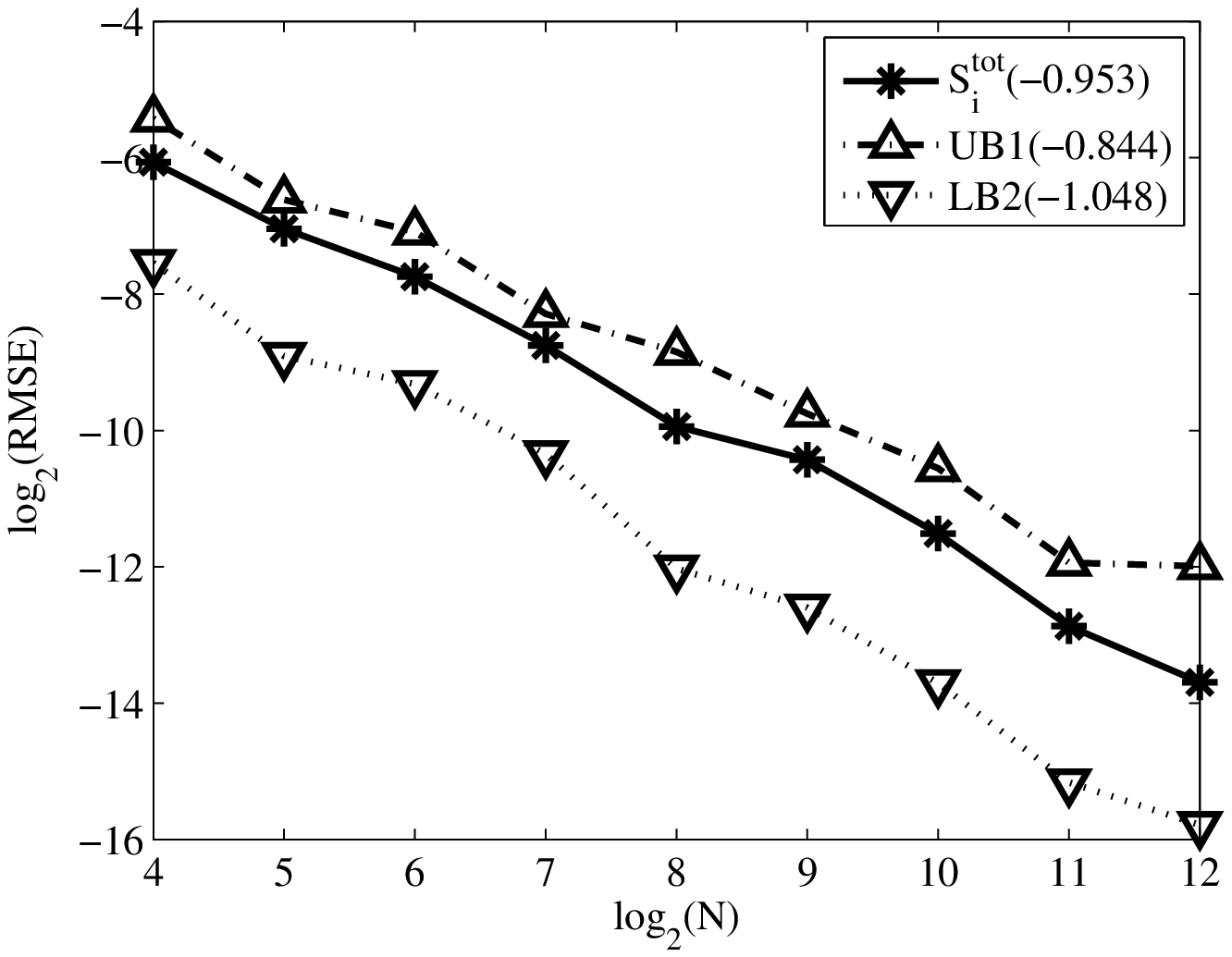} }
\label{Fig.sub.2}
\subfigure [] {\includegraphics[width=3.5cm,clip]{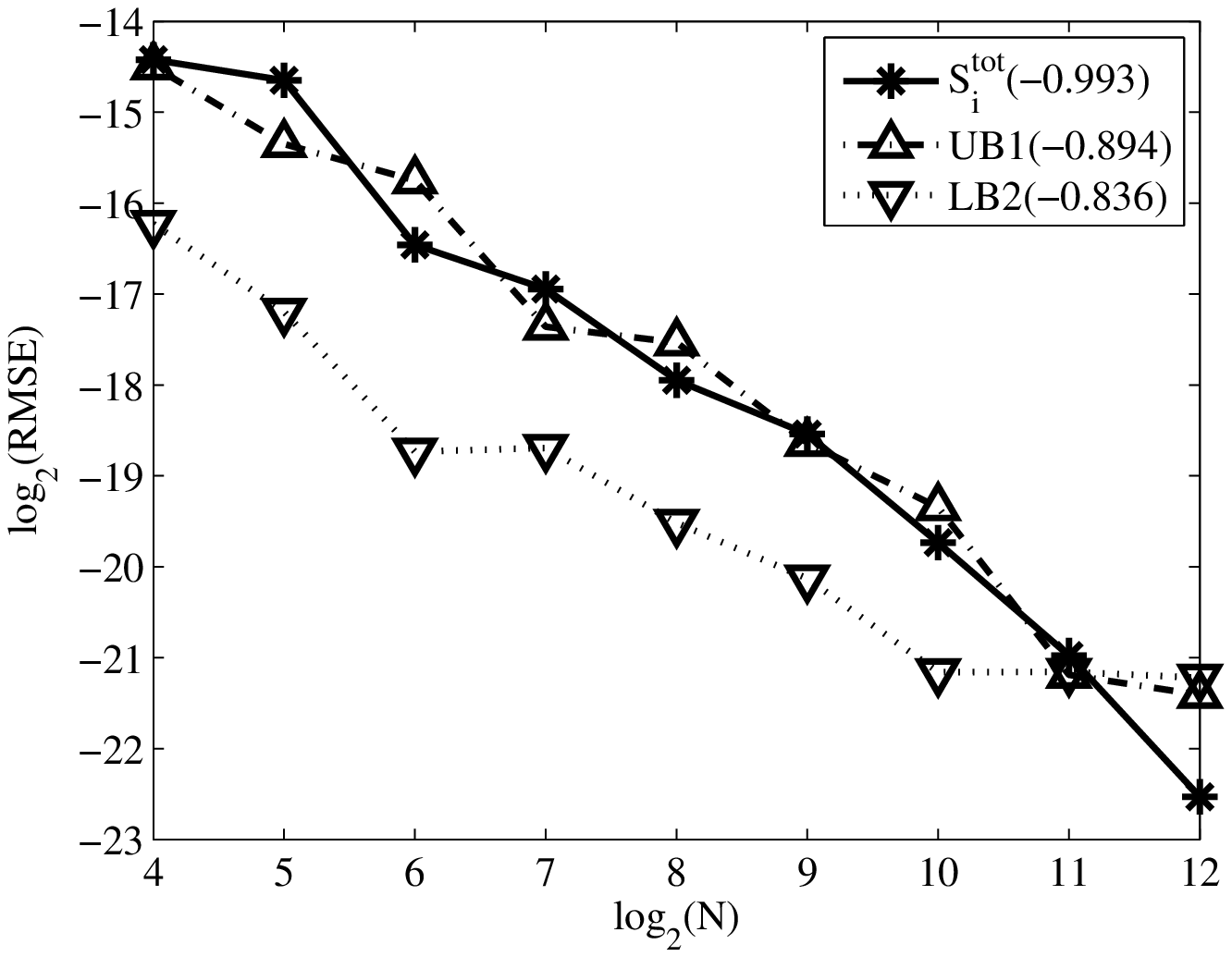} }
\label{Fig.sub.3}
}
\label{fig_sim_2}
\caption{RMSE of $S_i^{tot}$, UB
and LB2 versus the number of sampled points. Example 2,
\textbf{\textit{a}}=[0,1,4.5,9,99,99,99,99], $d$=8. Variable 1 (a), variable
3 (b) and variable 5 (c).}
\end{figure}
\bigskip
\textbf{Example 3. }Hartmann function $f({\bf{x}}) =  - \sum\limits_{i = 1}^4 {c_i \exp \left[ { - \sum\limits_{j = 1}^n {\alpha _{ij} (x_j  - p_{ij} )^2 } } \right]} $, $x_i  \in [0,1]$. For this test case a relationship between the values LB1, LB2 and
$S_i$ varies with the
change of input (Table 4, Figure 3): for variables $x_2$ and $x_6$ LB1$ > $
$S_i$$ > $
LB2, while for all other variables LB1$ < $
LB2 $ < $$S_i$. LB* is much smaller than
$S_i^{tot}$ for all inputs.
Values of $m$* also vary with the change of input. For all variables but variable 2
UB1 $ > $ UB2.

\begin{figure}[!]
\centering
\includegraphics[scale=0.7]{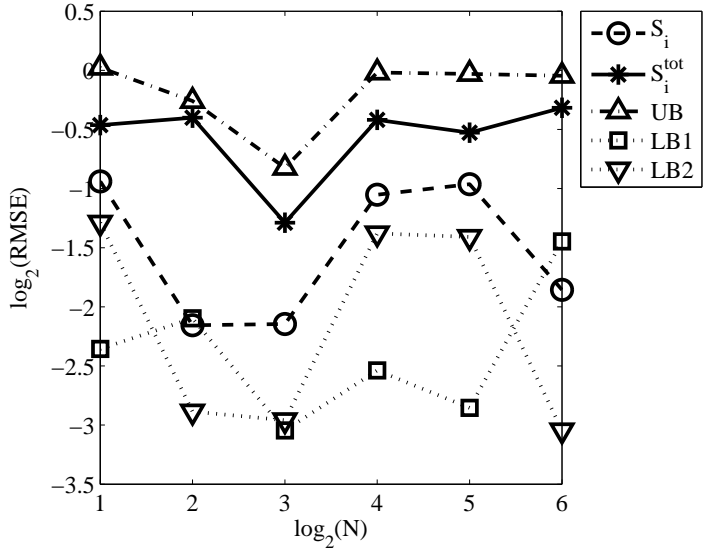}
\label{fig_sim_3}
\caption{Values of $S_i$,$S_i^{tot}$, UB1, LB1 and LB2
for all input variables. Example 3.}
\end{figure}

\begin{table}[!]
\caption{Values of $m^*$, LB1, LB2, UB1, UB2,$S_i$ and $S_i^{tot}$ for all input variables.}
\label{tab4}
\begin{tabular}
{|p{30pt}|p{45pt}|p{45pt}|p{45pt}|p{50pt}|p{50pt}|p{50pt}|}
\hline
$ $&
$x_1$&
$x_2$&
$x_3$&
$x_4$&
$x_5$&
$x_6$ \\
\hline
$LB1$&
$0.0044$&
$0.0080$&
$0.0009$&
$0.0029$&
$0.0014$&
$0.0357$\\
\hline
$LB2$&
$0.0515$&
$0.0013$&
$0.0011$&
$0.0418$&
$0.0390$&
$0.0009$\\
\hline
$m^*$&
$4.6$&
$10.2$&
$17.0$&
$5.5$&
$3.6$&
$19.9$\\
\hline
$LB^*$&
$0.0515$&
$0.0080$&
$0.0011$&
$0.0418$&
$0.0390$&
$0.0357$\\
\hline
$S_i$&
$0.115$&
$0.00699$&
$0.00715$&
$0.0888$&
$0.109$&
$0.0139$\\
\hline
$S_i^{tot}$&
$0.344$&
$0.398$&
$0.0515$&
$0.381$&
$0.297$&
$0.482$\\
\hline
$UB1$&
$1.089$&
$0.540$&
$0.196$&
$1.088$&
$1.073$&
$1.046$\\
\hline
$UB2$&
$1.051$&
$0.550$&
$0.150$&
$0.959$&
$0.932$&
$0.899$\\
\hline
\end{tabular}
\end{table}

\section{Conclusions}\label{sec:6}

We can conclude that using lower and upper bounds based on DGSM it is
possible in most cases to get a good practical estimation of the values of
$S_i^{tot}$ at a fraction
of the CPU cost for estimating $S_i^{tot}$. Small values of upper bounds imply small values of
$S_i^{tot}$. DGSM can be used
for fixing unimportant variables and subsequent model reduction. For linear
function and product function, DGSM can give the same variable ranking as
$S_i^{tot}$. In a general case
variable ranking can be different for DGSM and variance based methods. Upper
and lower bounds can be estimated using MC/QMC integration methods using the
same set of partial derivative values. Partial derivatives can be efficiently
estimated using algorithmic differentiation in the reverse (adjoint) mode.

We note that all bounds should be computed with sufficient accuracy. Standard
techniques for monitoring convergence and accuracy of MC/QMC estimates
should be applied to avoid erroneous results.

\[
\]

\begin{acknowledgement}
The authors would like to thank Prof. I. Sobol' his invaluable contributions
to this work. Authors also gratefully acknowledge the financial support by
the EPSRC grant EP/H03126X/1.
\end{acknowledgement}

%
\bibliographystyle{spmpsci}
\bibliography{mybibfile}
%

\end{document}